\documentclass[12pt]{amsart}
\usepackage[latin1]{inputenc}
\usepackage{amsmath}
\usepackage{amssymb}
\usepackage{amsthm}
\usepackage{xcolor}
\usepackage{graphicx}
\usepackage{rotating}
\newtheorem{definition}{Definition}[section]

\newtheorem{lemma}[definition]{Lemma}
\newtheorem{theorem}[definition]{Theorem}
\newtheorem{prop}[definition]{Proposition}

\numberwithin{equation}{section}

\begin{document}

\title[Characteristic polynomials of products of Wigner matrices]{\sc Characteristic polynomials of products of Wigner matrices: finite-$N$ results and Lyapunov universality}

\author{Gernot Akemann, Friedrich G\"otze and Thorsten Neuschel}

\address{Faculty of Mathematics and Faculty of Physics, Bielefeld University, PO-Box 100131, 33501 Bielefeld, Germany}
%\email{akemann@physik.uni-bielefeld.de}

\keywords{Averages of characteristic polynomials, products of random matrices, Wigner matrices, Lyapunov exponents, universality.}
\commby{}
\begin{abstract}
We compute the average characteristic polynomial of the hermitised product of $M$ real or complex Wigner matrices of size $N\times N$ and the average of the characteristic polynomial of a product of $M$ such Wigner matrices 
times the  characteristic polynomial of the  conjugate matrix.
Surprisingly, the results agree with that of the product of $M$ real or complex Ginibre matrices at {\it finite}-$N$, which have i.i.d. Gaussian entries. For the latter the average 
characteristic polynomial yields the orthogonal polynomial for the singular values of the product matrix, whereas the product of the two characteristic polynomials involves the kernel of complex eigenvalues.
This extends the result of Forrester and Gamburd for one characteristic polynomial of a single random matrix and only depends on the first two  moments. 
In the limit $M\to\infty$ at fixed $N$ we determine the locations of the zeros of a single characteristic polynomial, rescaled as Lyapunov exponents by taking the logarithm of the $M$th root. 
The position of the $j$th zero 
agrees asymptotically for large-$j$ with the position of the $j$th Lyapunov 
exponent for products of Gaussian random matrices, 
hinting at the universality of the latter.

\end{abstract}

\maketitle
%%%%%%%%%%%%%%%%%%%%%%%%%%%%%%%%%%%%%%%%%%%%%%%%%
\section{Introduction and Main Results}

Characteristic polynomials represent one of the central building blocks when studying the spectral statistics of random matrices. For example,  in invariant ensembles the Heine-formula directly relates the expectation value of a single characteristic polynomial over a random matrix of size $N\times N$ to the orthogonal polynomial of degree $N$. Invariant ensembles represent determinantal point processes, and the corresponding kernel of orthogonal polynomials follows from the expectation value of two characteristic polynomials \cite{PZJ,BorodinOP} at finite-$N$. This statement extends to non-Hermitian ensembles as well \cite{AV}. 
In applications of random matrices characteristic polynomials also play a key role, e.g. in comparison to moments and correlations  of the Riemann $\zeta$-function \cite{KN00,K}, or in the theory of strong interactions in the presence of baryon chemical potential \cite{Misha}.

One of the central questions in random matrix theory is that of universality, that is the independence of the distribution of matrix elements in asymptotic regimes such as the limit of large matrix size.  Two main classes of deformations of the classical ensembles with independent Gaussian distribution of matrix elements exist:  Wigner ensembles, where the independence is kept and invariance is dropped, allowing for more general distributions than Gaussian, and the invariant ensembles, where independence is dropped while keeping invariance under orthogonal or unitary transformations. This introduces a dependence among matrix elements, typically through a single potential in the distribution. 

Given that invariant ensembles represent determinantal point processes at finite-$N$, the knowledge of the kernel at finite-$N$ allows a direct asymptotic analysis
of the marginals or $k$-point correlation functions of the matrix eigenvalues.
Here, sophisticated techniques as the Riemann-Hilbert method have been developed. 
The universality of products and ratios of characteristic polynomials has been directly addressed as well, yielding a generating functional for the kernel, both for invariant \cite{BH,FS1} and Wigner ensembles \cite{GK}, see also  
\cite{Tatyana, Afanasiev} for recent work using supersymmetry. 
We refer to \cite{FS,BDS} for most concise expressions for averages of products and ratios of characteristic polynomials at finite-$N$, and to \cite{GGK} for the supersymmetric perspective on that.

In Wigner ensembles, however, such determinantal structures seem to be completely absent at finite matrix size. In consequence powerful probabilistic tools have been developed by several groups, in oder to prove universality in the various scaling regimes, for Hermitian and non-Hermitian random matrices, cf. \cite{Erdoes,TaoVu}, respectively and references therein.  
How can we understand this broad universality? Are there perhaps also objects within Wigner ensembles, that show a similar structure as for Gaussian ensembles at finite-$N$?  
Indeed it was shown by Forrester and Gamburd \cite{FG04}, that the expectation value of a single characteristic polynomial of Wigner matrices of size $N\times N$ agrees with that of the corresponding Gaussian ensemble. It is given by the Hermite polynomial for the Gaussian Unitary Ensemble (GUE) and by the Laguerre polynomial for the complex Wishart ensemble 
(also called chiral GUE or Laguerre unitary ensemble). The same polynomials are obtained for real Wigner matrices \cite{FG04}.
In this short article we will extend the list of such examples of an exact agreement at finite-$N$ to products of $M$ 
Wigner matrices, both for singular values and complex eigenvalues of the product matrix. 

Consider $M$ independent, Gaussian random matrices $G_1,\ldots, G_M$ of size $N\times N$. Each matrix $G_k$ has independent  matrix elements $g_{i,j}^{(k)}$ with identical normal distribution, with zero mean and variance $\sigma^2_k>0$, $\mathbb{E}[g_{i,j}^{(k)}] =0$ and $\mathbb{E}[g_{i,j}^{(k)}\overline{g_{m,n}^{(l)}}] =\delta_{i,m}\delta_{j,n}\delta_{k,l}\sigma_k^2$. 
The squared singular values of $G_k$ are called Wishart ensembles, whereas the complex eigenvalues of $G_k$ are called Ginibre ensembles.

In \cite{AIK} it was shown for complex matrices with unit variances $\sigma_k=1$ that the squared singular values of the product matrix $G_1\cdots G_M$ form a determinantal point process, representing an example for a polynomial ensemble. The corresponding kernel of biorthogonal functions was explicitly determined in \cite{AIK} using Gram-Schmidt orthogonalisation\footnote{It is not difficult to extend the proof in \cite{AIK} to allow for $\sigma_1,\ldots,\sigma_M>0$.}. As the Heine formula trivially extends to polynomial ensembles, the following 
holds for the orthogonal polynomials:
\begin{equation}
\label{M-Wishart-prod}
\mathbb{E}\left[\det\left[xI_N-(G_1\cdots G_M)^*(G_1\cdots G_M)\right]\right]=p_N^{(M)}(x)\ ,
\end{equation}
where the polynomial $p_N^{(M)}$ of degree $N$ is given by Eq. (40) \cite{AIK}
\begin{equation}
\label{OP}
p_N^{(M)}(x)=
(-1)^N (N!)^{M+1}\sum_{k=0}^N 
\frac{(-x)^k}{(N-k)!(k!)^{M+1}}\ . 
\end{equation}
For $M=1$ this relation is well known to hold for a single Wishart ensemble, where the polynomial reduces to the Laguerre polynomial $p_N^{(M=1)}(x)=(-1)^NN!\,L_N(x)$, in monic normalisation. 
We find that the same relation extends to the product of real or complex independent 
Wigner matrices, where we allow for non homogeneous variances.

\begin{theorem}
\label{1polynomial}
Let $X_1,\ldots,X_M$ be $M$ independent Wigner matrices of size $N\times N$ such that the entries $x_{i,j}^{(k)}$ of every matrix $X_k$ are independent as well, having arbitrary real or complex
distributions with zero mean and variance $\sigma_k>0$, i.e., $\mathbb{E}\left[x_{i,j}^{(k)}\right] =0$ and $
\mathbb{E}[x_{i,j}^{(k)}\overline{x_{m,n}^{(l)}}] =\delta_{i,m}\delta_{j,n}\delta_{k,l}\sigma_k^2$.
Defining $\tau_M=\sigma_1\cdots\sigma_M$, the expectation values for the following characteristic polynomials read 
\begin{eqnarray}
\label{M-Wigner-prod}
\mathbb{E}\left[\det\left[xI_N-\left(X_1 \cdots X_M\right)^* \left(X_1\cdots X_M\right)\right]\right]&=& \tau_M^{2N}
p_N^{(M)}\left(\frac{x}{\tau_M^2}\right)\ ,\\
\label{M-mixed-prod}
\mathbb{E}\left[\det\left[xI_N-\left(X_1^{*}X_1\right)\ldots \left(X_M^{*}X_M\right)\right]\right]&=& \tau_M^{2N}
p_N^{(M)}\left(\frac{x}{\tau_M^2}\right)\ .
\end{eqnarray}
\end{theorem}

In particular, Eq. \eqref{M-Wigner-prod} agrees with the expression for the product of complex Ginibre matrices \eqref{M-Wishart-prod}. Based on bosonisation it was shown in \cite{Mario15} that the same result \eqref{M-Wigner-prod}
holds for products of real Ginibre matrices, treating all symmetry classes in a unified way. 
Theorem \ref{1polynomial} extends the result  by Forrester and Gamburd  for a single Wigner matrix $M=1$, see 
\cite[Prop. 12]{FG04}.
A similar result holds for a single real symmetric or Hermitian Wigner matrix $H$ with variance $\sigma^2$, cf. 
\cite[Prop 11.]{FG04} 
\begin{equation}
\label{Hermite}
\mathbb{E}\left[\det\left[xI_N-H\right]\right]=\left(\frac{\sigma}{2}\right)^{-N}H_N\left(\frac{x}{\sigma}\right)\ ,
\end{equation}
with a monic Hermite polynomial on the right-hand side. In particular this formula holds for Gaussian random matrices.  
It is tempting to expect that Theorem \ref{1polynomial} implies similar  
combinatorial consequences as \cite{FG04} for $M=1$, in particular as the moment generating function of products of random matrices relates to Fuss-Catalan numbers, see e.g. \cite{Neuschel1}.

The fact that the two equations in Theorem \ref{1polynomial} agree extends the observed asymptotic commutativity of (rectangular) matrices $G_k$  for $N\to\infty$  described first in \cite{BurdaLivanSwiech}, and later on for finite-$N$ in \cite{IK} in a weak sense. 

We turn to the complex eigenvalues and corresponding characteristic polynomials of products of real or complex Ginibre and Wigner matrices. Due to the independence of the matrices and matrix elements, for both products of Ginibre and Wigner matrices the expectation value of a single characteristic polynomial is trivial,
\begin{equation}
\mathbb{E}\left[\det\left[zI_N-(G_1\cdots G_M)\right]\right]=
\mathbb{E}\left[\det\left[zI_N-(X_1\cdots X_M)\right]\right]=z^N\ .
\end{equation}
For products of independent complex  respectively real Ginibre matrices it was shown in \cite{ABu} respectively \cite{FI} that the complex eigenvalues of the product matrix $G_1\cdots G_M$ form a determinantal respectively Pfaffian point process, with a rotationally invariant weight function. 
Thus the monomials $z^N$ are the orthogonal polynomials respectively the even subset of skew orthogonal polynomials as well, albeit trivial ones. The determinantal point process is of orthogonal polynomial type, being proportional to the modulus square of the Vandermonde determinant of complex eigenvalues. Thus the kernel is given by a sum over orthonormalised polynomials, containing nontrivial information about the weight through their (squared) norms $h_k$. A similar statement holds for the kernel of skew orthogonal polynomials for the complex eigenvalues of the Pfaffian point process  \cite{FI}.

In orthogonal polynomial ensembles it is known, both for real \cite{PZJ,BorodinOP} and complex eigenvalues \cite{AV}, that the kernel can be expressed 
by a product of two characteristic polynomials. The same holds true for the kernel of skew orthogonal polynomials in Pfaffian process including the real Ginibre ensemble \cite{APS}.
Combined with the result for the kernel $K_N$ of products of $M$ independent complex Ginibre matrices \cite{ABu}, extended to non homogeneous variances in \cite{ACi}, we have the following statement
\begin{equation}
\label{M-Ginibre-kernel}
\mathbb{E}\left[\det\left[zI_N-G_1\cdots G_M\right]\det\left[wI_N-(G_1\cdots G_M)^*\right]\right]= \tau_M^{2N} h_N^{(M)} K_{N+1}^{(M)}\left(\frac{z}{\tau_M},\frac{w}{\tau_M}\right) ,
\end{equation}
%with 
\begin{equation}
\label{kernel-def}
K_{N+1}^{(M)}(z,w)=\sum_{k=0}^N  \frac{(zw)^k}{h_k^{(M)}}\ , \ \ h_k^{(M)}=\pi(k!)^M\ .
\end{equation}
For products of $M$ real Ginibre matrices the right-hand side of \eqref{M-Ginibre-kernel} is proportional to  
the anti-symmetric kernel of skew orthogonal polynomials $\kappa_{N}(z,w)/(z-w)$,  for $z\neq w$ \cite{APS,FI}. 
We will show that the same result holds for products of independent Wigner matrices.
\begin{theorem}
\label{2polynomial}
Given $M$ independent Wigner matrices satisfying  the same conditions as in Theorem \ref{1polynomial}, the average of two characteristic polynomials with conjugate matrices reads: 
\begin{equation}
\label{M-Wigner-kernel}
\mathbb{E}\left[\det\left[zI_N- \left(X_1\cdots X_M\right)\right]\det\left[wI_N-\left(X_1 \cdots X_M\right)^*\right]\right]=\tau_M^{2N}h_N^{(M)} K_{N+1}^{(M)}\left(\frac{z}{\tau_M},\frac{w}{\tau_M}\right).
\end{equation}
\end{theorem}

It is simple to understand why the above results do not easily extend to more products of characteristic polynomials: In all averages \eqref{M-Wigner-prod}, \eqref{M-mixed-prod} and \eqref{M-Wigner-kernel} every matrix $X_k$ and its adjoint $X_k^*$ appear exactly once. This also explains the absence of higher moments of these matrices. 
In principle, the agreement between Gaussian and Wigner matrices at finite-$N$ could thus be extended to the expectation of any polynomial that shares this property.
Let us emphasise that the proofs of Theorem \ref{1polynomial} and \ref{2polynomial} are purely algebraic and constructive. They directly yield the explicit combinatorial result for Wigner matrices,   
without recurring to independent calculations for Gaussian matrix elements. 
For simplicity we have restricted ourselves to square matrices, see e.g.  \cite{AIK} for the generalisation of \eqref{OP} to products of rectangular complex Gaussian matrices. We expect that this agreement holds for products of rectangular Wigner matrices as well.

Since the above identities between Gaussian and Wigner ensembles already hold for finite-$N$, the universality of these expectations in various large-$N$ limits is guaranteed. Let us emphasise, however, that this does not imply an identity for all $k$-point singular value or complex eigenvalue correlation functions  at finite-$N$,  as then Wigner ensembles do not possess any determinantal or Pfaffian structure.
For singular values of complex matrices, the derivation of the kernel within polynomial ensembles requires to evaluate the expected ratio of two characteristic polynomials, see \cite{ForresterDesrosier} for more details. The evaluation of such objects remains a highly nontrivial task for products of Wigner ensembles. 
While for complex eigenvalues \eqref{M-Ginibre-kernel} indeed establishes the (skew-)kernel, the $k$-point correlation functions also depend on the weight function multiplying this (skew-)kernel, which contributes non-trivially in the large-$N$ limit.

In order to derive a non-trivial universality statement for products of Wigner matrices based on the above findings, we consider a growing number of factors, choosing $M\to\infty$, while keeping $N$ fixed.
Consider the zeros of the average characteristic polynomial \eqref{M-Wigner-prod}
\begin{equation}
\label{zerosOP}
\tau_M^{2N}p_N^{(M)}\left(\frac{x}{\tau^2_M}\right) = \prod_{j=1}^N (x-z_j) \ ,
\end{equation}
denoted by  $z_j=z_j^{(M)}$ in increasing order. These are all non-negative as is shown in \cite{Neuschel1}.

We wish to compare these zeros to the limiting Lyapunov exponents of the product matrix $(X_1 \cdots X_M)^*(X_1 \cdots X_M)$. They are defined in terms of the ordered non-negative eigenvalues (or squared singular values) $\lambda_1^{(M)},\ldots,\lambda_N^{(M)}$ of the product matrix.
We first define the \emph{incremental} Lyapunov exponents by the following re-scaled quantities 
\begin{equation}\label{Lincrement}
\mu_j^{(M)}:=\frac{1}{2M}\log\left( \lambda_j^{(M)}\right),\quad j=1,\ldots,N.
\end{equation}
The Lyapunov exponents are obtained in the limit 
\begin{equation}
\mu_j = \lim_{M\to\infty}\mu_j^{(M)}, \quad j=1,\ldots,N\ ,
\end{equation}
and we refer to \cite{reviewL} for the vast literature about their existence for Gaussian and other random matrices. 
In the same re-scaling as in \eqref{Lincrement} we obtain the following for the zeros. 
 \begin{theorem}
 \label{Thmzeros}
For the ordered zeros $z_j^{(M)}$ of the averaged characteristic polynomials \eqref{M-Wigner-prod} of the product of $M$ Wigner matrices with variances $\sigma_k>0$, satisfying $\lim_{M\to\infty} \tau_M^{1/M} =\sigma >0$, it holds
\begin{equation}
\label{zerolim}
\lim_{M\to\infty} \frac{1}{2M}\log z_j^{(M)}= \frac{1}{2}\left(\log( j)+ \log(\sigma^2)\right),\quad j=1,\ldots,N.
\end{equation}
 \end{theorem}
In the case that all matrices $X_j$ are real or complex Ginibre matrices, labelled by $\beta=1,2$ respectively, the corresponding Lyapunov exponents are explicitly known \cite{Newman,Peter} 
\begin{equation}
\mu_j = \frac{1}{2} \left(\Psi\left(\frac{\beta j}{2}\right)+ \log\left(\frac{2\sigma^2}{\beta}\right)\right),\quad j=1,\ldots,N,
\end{equation}
where $\Psi$ denotes the Digamma function. We refer to \cite{JesperLyap} for rectangular matrices. For large $j$, applying the large argument asymptotic for  the Digamma-function \cite{NIST}, we get
\begin{equation}
\Psi(j)=\log(j)+\mathcal{O}\left(\frac{1}{j}\right)\ .
\end{equation}
In consequence, for large orders $j\leq N$ with $N$ large but not necessarily infinite, the $\beta$-dependence drops out and the 
$j$th Lyapunov exponent and $j$th zero of the characteristic polynomial (which is universal) agree up to 
an error $\mathcal{O}(1/j)$. This strongly suggests
that the $j$th Lyapunov exponents for products of Wigner matrices also become universal.

%%%%%%%%%%%%%%%%%%%%%%%%%%%%%%%%%%%%%%%%%%%%%%%%%
\section{Finite-$N$ identity for Wigner and Ginibre matrices:\\ Proof of Theorems \ref{1polynomial} and \ref{2polynomial}} 

The proofs use the independence and simple linear algebra, including the Cauchy-Binet formula. For convenience of notation, let us denote the set of all subsets of $\{1,\ldots, N\}$ with exactly $r$ elements by $\mathcal{K}_{r,N}$. For $K= \{k_1 < \ldots < k_r\}, L=\{\ell_1 < \ldots < \ell_r\} \in \mathcal{K}_{r,N}$ and a matrix $X$ of size $N \times N$ we write the determinant of the corresponding $r\times r$ sub-matrix as follows:
\begin{equation}
\det \left[\begin{matrix} & L\\
K & X
\end{matrix}\right]=\det \left[ \begin{matrix}
x_{k_1,\ell_1}& \ldots& x_{k_1,\ell_r}\\
\vdots& \ddots&\vdots\\
x_{k_r,\ell_1}& \ldots& x_{k_r,\ell_r}\\
\end{matrix} \right] = \det[x_{i,j}]_{i\in K;\, j\in L}
\end{equation}
Thus  the sub-matrix on the right-hand side is obtained from $X$ by choosing the rows with indices $1\leq k_1 < \ldots < k_r \leq N$ and then the columns  
with indices $1\leq \ell_1 < \ldots < \ell_r \leq N$. We begin by introducing  the following Lemma about expectations of two determinants of different sub-matrices of equal size of the same matrix and its adjoint.

\begin{lemma}\label{expectedminorprod}
For $K, \tilde{K}, L, \tilde{L} \in \mathcal{K}_{r,N} $ and $j=1\ldots,N$ we have
\begin{equation}
\label{La1.1}
 \mathbb{E} \left[ \det \left[\begin{matrix} & \tilde{K}\\
	K & X_j
	\end{matrix}\right] \det  \left[\begin{matrix} & L\\
	\tilde{L} & X^*_j
	\end{matrix}\right] \right] =\begin{cases}
	r!\,\sigma_j^{2r}~,\quad \text{if}\quad K=L \quad \text{and}\quad \tilde{K}= \tilde{L} ,\\
	0~, \quad\quad\ \  \text{else}. 
	\end{cases}
\end{equation}	
Moreover, it holds
\begin{equation}
\label{La1.2}
\mathbb{E}  \det \left[\begin{matrix} & L\\
	K & X_j^* X_j 
	\end{matrix}\right]=\begin{cases}
	\frac{N!}{(N-r)!}\,\sigma_j^{2r}~,\quad \text{if}\quad K=L , \\
	0~, \quad \quad\quad\quad\  \text{else}. 
	\end{cases}
\end{equation}	
\end{lemma}
\begin{proof}
	We have
\begin{eqnarray}
&&	\mathbb{E} \left[ \det \left[\begin{matrix} & \tilde{K}\\
	K & X_j
	\end{matrix}\right] \det  \left[\begin{matrix} & L\\
	\tilde{L} & X^*_j
	\end{matrix}\right] \right] = \mathbb{E} \left[ \det \left[\begin{matrix} & \tilde{K}\\
	K & X_j
	\end{matrix}\right] \overline{\det  \left[\begin{matrix} & \tilde{L}\\
		L & X_j
		\end{matrix}\right]} \right]
		\nonumber\\
		&&= \sum_{\pi, \sigma \in S_{r}} \mathrm {sign} (\pi) ~  \mathrm {sign} (\sigma)~ \mathbb{E}\left[
		\prod_{n=1}^r x^{(j)}_{k_n,\pi(\tilde{k}_n)}\prod_{m=1}^r \overline{x^{(j)}_{\ell_m,\sigma(\tilde{\ell}_m)}}\right].
		\label{E2dets}
\end{eqnarray}	

	First, we assume that $K\neq L$ or $\tilde{K} \neq \tilde{L}$. In writing both determinants as sums according to Leibniz' rule and expanding the product, every summand of the resulting sum is a product of entries  of $X_j$ {\it without repetition}. That means, due to $K\neq L$ or $\tilde{K} \neq \tilde{L}$ there will be at least one mismatch in the first or second index pairs.
  Using the independence of all entries, we conclude that every summand vanishes in expectation,  so that we obtain the second case in \eqref{La1.1}.

	Now let us assume $K=L$ and $\tilde{K}=\tilde{L}$. In this case we can write \eqref{E2dets} as 
	\begin{align*}&\mathbb{E} \left[ \det \left[\begin{matrix} & \tilde{K}\\
	K & X_j
	\end{matrix}\right] \det  \left[\begin{matrix} & K\\
	\tilde{K} & X^*_j
	\end{matrix}\right] \right]
	=\sum_{\pi \in S_{r}} \mathbb{E}\left[\left\vert x^{(j)}_{k_1,\pi(\tilde{k}_1)}\right\vert^2\right] \cdots  \mathbb{E}\left[\left\vert x^{(j)}_{k_r,\pi(\tilde{k}_r)}\right\vert^2\right]
	=r!\,\sigma_j^{2r}\ ,
	\end{align*}
	where we used that all summands with permutations $\pi \neq \sigma$ vanish. This shows Eq. \eqref{La1.1}. 
	
	To see the second part Eq. \eqref{La1.2}, we write $\mathcal{N} = \{1,\ldots,N\}$ and observe
\begin{equation*}
\mathbb{E}  \det \left[\begin{matrix} & L\\
	K & X_j^* X_j 
	\end{matrix}\right]= \mathbb{E}  \det \left[\begin{pmatrix} & \mathcal{N}\\
	K & X_j^* 
	\end{pmatrix}\begin{pmatrix} & L\\
	\mathcal{N} & X_j 
	\end{pmatrix}\right].
\end{equation*}	
Notice that on the left-hand side the sub-matrix of $X_j^*X_j$ is of size $r\times r$, whereas the matrices multiplied on the right-hand side are of sizes $r\times N$ and $N\times r$, respectively.
	For a set of indices $1\leq \nu_1 < \ldots < \nu_r \leq N$ we write $V=\{\nu_1,\ldots,\nu_r\}$, and an application of the Cauchy-Binet formula gives
 \begin{equation}\label{CBapp}
 	\mathbb{E}  \det \left[\begin{matrix} & L\\
	K & X_j^* X_j 
	\end{matrix}\right]=  \sum_{1\leq \nu_1 <\ldots < \nu_r \leq N} \mathbb{E}  \det \left[\begin{pmatrix} & V\\
	K & X_j^* 
	\end{pmatrix}\begin{pmatrix} & L\\
	V & X_j 
	\end{pmatrix}\right].
 \end{equation}
By part one of the lemma, Eq. \eqref{La1.1}, we know that for every summand 
with set of indices $K,V$, and $V,L$ 
we obtain a non-vanishing contribution $r!\, \sigma_j^{2r}$ for $K=L$ only, and thus  
\begin{equation*}	
	\mathbb{E}  \det \left[\begin{matrix} & K\\
	K & X_j^* X_j 
	\end{matrix}\right]= \binom{N}{r} r!\, \sigma_j^{2r}= \frac{N!}{(N-r)!}\, \sigma_j^{2r}\ .
	\end{equation*}
	\end{proof}
Now we can turn to the derivation of the explicit expression for the average of a single characteristic polynomial.
\begin{proof}[Proof of Theorem \ref{1polynomial}]
	We start by deriving the expression given in the first statement of Eq. \eqref{M-Wigner-prod}. Therefore, we expand the characteristic polynomial into powers of $x$ by expressing the corresponding coefficients in terms of the principal minors 
		\begin{align*}
	&\det\left[xI_N- X_M^*\cdots X_1^* X_1\cdots X_M \right] \\
	&
	= \sum_{\nu=0}^N (-1)^{N-\nu} x^{\nu} \sum_{K^{(1)} \in \mathcal{K}_{N-\nu,N}} \det \left[\begin{matrix} & K^{(1)}\\
	K^{(1)} & X_M^*\cdots X_1^* X_1\cdots X_M
	\end{matrix}\right].  
 	\end{align*}
	The minors of the product on the right-hand side can be expanded by means of an iterated application of the Cauchy-Binet formula \eqref{CBapp}. This way we obtain
	\begin{align*}
	&\det\left[xI_N- X_M^*\cdots X_1^* X_1\cdots X_M \right] \\
	&= \sum_{\nu=0}^N (-1)^{N-\nu} x^{\nu} \sum_{K^{(1)},\ldots, K^{(2M)}  \in \mathcal{K}_{N-\nu,N}}  \det \left[\begin{matrix} & K^{(2)}\\
	K^{(1)} & X_M^*
	\end{matrix} \right] \det \left[\begin{matrix} & K^{(3)}\\
	K^{(2)} & X_{M-1}^*
	\end{matrix}\right] \cdots\\
	& \hspace{5cm} 
	\times  \det \left[\begin{matrix} & K^{(M+1)} \\
	K^{(M)} & X_1^*
	\end{matrix}\right] \det \left[\begin{matrix} & K^{(M+2)} \\
	K^{(M+1)} & X_1
	\end{matrix}\right]
	\\&\hspace{5cm} \cdots \times  
	\det \left[\begin{matrix} & K^{(2M)} \\
	K^{(2M-1)} & X_{M-1}
	\end{matrix}\right] \det \left[\begin{matrix} & K^{(1)} \\
	K^{(2M)} & X_{M}
	\end{matrix}\right].
	\end{align*}
	We reorder the factors in the product of the $2M$ determinants on the right-hand side 
by pairing matrices $X_j^*$ and $X_j$ and take the expectation to obtain
\begin{align*}
&\mathbb{E} \det\left[xI_N- X_M^*\cdots X_1^* X_1\cdots X_M \right] \\
&= \sum_{\nu=0}^N (-1)^{N-\nu} x^{\nu} \sum_{K^{(1)},\ldots, K^{(2M)}  \in \mathcal{K}_{N-\nu,N}} \mathbb{E} \left[ \det \left[\begin{matrix} & K^{(2)}\\
K^{(1)} & X_M^*
\end{matrix} \right] \det \left[\begin{matrix} & K^{(1)}\\
K^{(2M)} & X_M
\end{matrix} \right] \right] \\
&\hspace{5cm} \times  \mathbb{E} \left[ \det \left[\begin{matrix} & K^{(3)}\\
K^{(2)} & X_{M-1}^*
\end{matrix} \right] \det \left[\begin{matrix} & K^{(2M)}\\
K^{(2M-1)} & X_{M-1}
\end{matrix} \right] \right] \\
&\hspace{5cm} \times \cdots \times \mathbb{E} \left[ \det \left[\begin{matrix} & K^{(M)}\\
K^{(M-1)} & X_2^*
\end{matrix} \right] \det \left[\begin{matrix} & K^{(M+3)}\\
K^{(M+2)} & X_2
\end{matrix} \right]  \right] \\
&\hspace{5cm} \times \mathbb{E} \left[ \det \left[\begin{matrix} & K^{(M+1)}\\
K^{(M)} & X_{1}^*
\end{matrix} \right] \det \left[\begin{matrix} & K^{(M+2)}\\
K^{(M+1)} & X_{1}
\end{matrix} \right] \right],
\end{align*}
where we also used the independence of the matrices in order to distribute the expectation over pairwise matching factors. Now we are able
 to apply Lemma \ref{expectedminorprod}. Thus, we first see that a summand vanishes as soon as one of the conditions $K^{(j)} = K^{(2M+2-j)}$, $j=1,\ldots, M$, is not satisfied. 
Evaluating the expectations explicitly we obtain
\begin{align*}
&\mathbb{E} \det\left[xI_N- X_M^*\cdots X_1^* X_1\cdots X_M \right] \\
&= \sum_{\nu=0}^N (-1)^{N-\nu} x^{\nu} \sum_{K^{(1)},\ldots, K^{(M+1)}  \in \mathcal{K}_{N-\nu,N}} \left((N-\nu)!\right)^M(\sigma_M\cdots\sigma_1)^{2(N-\nu)}\\
&= \sum_{\nu=0}^N (-1)^{N-\nu} x^{\nu} \binom{N}{N-\nu}^{M+1} \left((N-\nu)!\right)^M\tau_M^{2(N-\nu)},
\end{align*}
which now leads to the claimed expression.

To derive the second average characteristic polynomial in Eq. \eqref{M-mixed-prod} we can proceed analogously. We first expand the polynomial in terms of principal minors, which in turn can be expanded by means of the Cauchy-Binet formula in the following way
\begin{align*}
&\mathbb{E}\left[
\det\left[xI_N-\left(X_1^{*}X_1\right)\ldots \left(X_M^{*}X_M\right)\right]
\right] 
\\
&= \sum_{\nu=0}^N (-1)^{N-\nu} x^{\nu} 
\sum_{K^{(1)},\ldots, K^{(M)}  \in \mathcal{K}_{N-\nu,N}}    
\mathbb{E} \left[
\det \left[\begin{matrix} & K^{(2)}\\
K^{(1)} & X_1^* X_1 
\end{matrix}\right] \ldots \det \left[\begin{matrix} & K^{(1)}\\
K^{(M)} & X_M^* X_M 
\end{matrix}\right]\right].
\end{align*}
We can use the independence of the matrices $X_j$ again in order to distribute the expectations over the factors, and observe that by Lemma \ref{expectedminorprod} we only have contributions from index sets with $K^{(1)}= K^{(2)} = \cdots = K^{(M)}$. Hence, we obtain
\begin{align*}
\mathbb{E}\left[\det\left[xI_N-\left(X_1^{*}X_1\right)\ldots \left(X_M^{*}X_M\right)\right]\right] &= \sum_{\nu=0}^N (-1)^{N-\nu} x^{\nu} \sum_{K^{(1)}\in \mathcal{K}_{N-\nu,N}} \left(\frac{N!}{\nu!}\right)^M\!(\sigma_1\cdots\sigma_M)^{2(N-\nu)}\\
&=\sum_{\nu=0}^N (-1)^{N-\nu} x^{\nu} \binom{N}{\nu} \left(\frac{N!}{\nu!}\right)^M\tau_M^{2(N-\nu)}, 
\end{align*}
from which the second statement Eq. \eqref{M-mixed-prod} follows.
\end{proof}

Let us turn to the average over two characteristic polynomials, related to the complex eigenvalues of the product matrix.
\begin{proof}[Proof of Theorem \ref{2polynomial}]
	We first deal with one of the determinants in Eq. \eqref{M-Wigner-kernel}, i.e., we have
	\begin{align*}
	\det\left[zI_N- X_1\cdots X_M\right] = \sum_{\nu=0}^N (-1)^{N-\nu} z^{\nu} \sum_{K^{(1)} \in \mathcal{K}_{N-\nu,N}} \det \left[\begin{matrix} & K^{(1)}\\
	K^{(1)} & X_1\cdots X_M
	\end{matrix}\right].  
	\end{align*}
	The minors of the product on the right-hand side can be expanded by means of an iterated application of the Cauchy-Binet formula \eqref{CBapp}. This way we obtain
	\begin{align*}
	&\det\left[zI_N- X_1\cdots X_M\right] \\
	&= \sum_{\nu=0}^N (-1)^{N-\nu} z^{\nu}\! \sum_{K^{(1)},\ldots, K^{(M)}  \in \mathcal{K}_{N-\nu,N}}\!  \det\! \left[\begin{matrix} & K^{(2)}\\
	K^{(1)} & X_1
	\end{matrix} \right]\! \det\! \left[\begin{matrix} & K^{(3)}\\
	K^{(2)} & X_2
	\end{matrix}\right]\cdots
	\det\! \left[\begin{matrix} & K^{(1)} \\
	K^{(M)} & X_M
	\end{matrix}\right]\!,
	\end{align*}
and 	similarly for $\det\left[wI_N- \left(X_1\cdots X_M\right)^*\right]$. We thus have for their average
	\begin{align*}
	&\mathbb{E}\left[\det\left[zI_N- X_1\cdots X_M\right] \det\left[wI_N- \left(X_1 \cdots X_M\right)^*\right]\right] \\
	& = \sum_{\nu, \mu = 0}^N (-1)^{\nu + \mu} z^{\nu} w^{\mu} \sum_{K^{(1)},\ldots, K^{(M)}  \in \mathcal{K}_{N-\nu,N}} \sum_{L^{(1)},\ldots, L^{(M)}  \in \mathcal{K}_{N-\mu,N}} \\
	& ~~~~
	\quad\mathbb{E} \left[ \det \left[\begin{matrix} & K^{(2)}\\
	K^{(1)} & X_1
	\end{matrix} \right] \det \left[\begin{matrix} & L^{(1)} \\
	L^{(2)} & X^*_1
	\end{matrix}\right] \right] \cdots \mathbb{E} \left[ \det \left[\begin{matrix} & K^{(1)}\\
	K^{(M)} & X_M
	\end{matrix} \right]  \det \left[\begin{matrix} & L^{(M)} \\
	L^{(1)} & X^*_M
	\end{matrix}\right]\right],
	\end{align*}
	using the independence of the matrices $X_1, \ldots, X_M$.
	If $\nu \neq \mu$, then for every index $j=1,\ldots, N-1$ we find that in one of the two matrices
	$\left[\begin{matrix} & K^{(j+1)} \\
	K^{(j)} & X_j
	\end{matrix}\right]$ and $\left[\begin{matrix} & L^{(j)} \\
	L^{(j+1)} & X^*_j
	\end{matrix}\right]$
    there are entries which do not appear in the other matrix.
	Thus, the expectation of the determinant of these two matrices  vanishes in this case, and using Lemma \ref{expectedminorprod}, we obtain, 
	\begin{align*}
	&\mathbb{E}\left[\det\left[zI_N- X_1\cdots X_M\right] \det\left[wI_N- \left(X_1 \cdots X_M\right)^*\right]\right] \\
	&= \sum_{\nu=0}^N (zw)^{\nu} \sum_{K^{(1)},\ldots, K^{(M)} \in \mathcal{K}_{N-\nu, N}} \left((N-\nu)!\right)^M(\sigma_1\cdots\sigma_M)^{2(N-\nu)}\\
	&=\sum_{\nu=0}^N (zw)^{\nu} \left((N-\nu)!\right)^M \binom{N}{N-\nu}^M\tau_M^{2(N-\nu)}.
	\end{align*}
\end{proof}

%%%%%%%%%%%%%%%%%%%%%%%%%%%%%%%%%%%%%%%%%%%%%%%%%
\section{Large-$M$ asymptotic of the zeros and Lyapunov spectrum}
\label{Lyapunov}

For simplicity, we first study the case  with unit variance $\sigma_k=1$ for $k=1,\ldots,M$.
We are interested in the behaviour for large $M$ and fixed dimensions $N$ of the suitably rescaled 
zeros of \eqref{zerosOP}
\begin{equation}
\frac{1}{2M}\log\left( z_j^{(M)}\right), \quad j=1,\ldots,N,
\end{equation}
since in this rescaling the zeros correspond to the incremental Lyapunov exponents \eqref{Lincrement}.
To this end, we study the behaviour of the quantities$ (z_j^{(M)})^{\frac{1}{2M}}, \quad j=1,\ldots,N,$
as $M\to \infty$. In order to achieve this we consider the accordingly rescaled polynomials
\begin{equation}
\mathcal{P}_N^{(M)}(z):=\sum_{k=0}^N \binom{N}{k}(-1)^k \left(\frac{z^{2k}}{k!}\right)^{M},\quad z\in \mathbb{C}.
\end{equation}
This rescaling introduces many additional zeros in the complex plane, however, this happens in a regular way and later we will be interested in comparing the positive zeros only to the incremental Lyapunov exponents.  

First, we study the asymptotic behaviour of the polynomials $\mathcal{P}_N^{(M)}(z)$ on the complex plane staying away from the circular domains 
\begin{equation}
C_{j,\epsilon} := \left\{z\in \mathbb{C}~: ~ j-\epsilon < \vert  z\vert^2 < j+\epsilon \right\}, \quad j=1,\ldots,N,
\end{equation}
for which we choose a fixed small $\epsilon >0$. In view of the facts that we have an analytically convenient explicit expression of the polynomials $\mathcal{P}_N^{(M)}$ and that the dimension $N$ remains fixed this can be done using elementary arguments and tools from complex analysis.
\begin{prop}\label{asymptoticsplane}
	We have for fixed $N$ and small $\epsilon >0$
	\begin{equation}\label{strongasymp}
	\mathcal{P}_N^{(M)}(z)=\binom{N}{\nu_z} (-1)^{\nu_z} \left(\frac{z^{2\nu_z}}{(\nu_z)!}\right)^M \left(1+\mathcal{O}\left(q^M\right) \right),\quad M\to\infty,\end{equation}
	uniformly in 
	$z\in \left(\mathbb{C}\cup \{\infty\}\right)\backslash \left(\cup_{j=1}^N C_{j,\epsilon} \right),$
	where $q:=q_{N,\epsilon}:=\frac{N}{N+\epsilon} \in (0,1)$ and the index $\nu_z$ is given by
\begin{equation}
\nu_z := \begin{cases}
	\lfloor \vert z \vert^2 \rfloor , \quad \vert z \vert^2 < N,\\
	N, \quad\quad\ \,  \vert z\vert^2 >N.
	\end{cases}
\end{equation}	
\end{prop}
\begin{proof}
	We note that the sequence $\frac{\vert z \vert^{2k}}{k!}, k=0,\ldots,N,$ is unimodal with a unique maximum at 
\begin{equation*}
 \begin{cases}
 0 , \quad \text{if}~ \vert z \vert^2 < 1,\\
 \nu, \quad \text{if}~ \nu < \vert z \vert^2 < \nu +1 \quad \text{for some}\quad \nu\in\{1,\ldots,N-1\},\\
 N, \quad \vert z\vert^2 >N.
 \end{cases}
\end{equation*} 
This suggests to consider the index $\nu_z$, so that this sequence 
strictly increases up to the index $\nu_z$, and strictly decreases afterwards.
 
 First, let us consider $\vert z\vert^2 \leq 1-\epsilon$, then we have 
\begin{equation*}
 \left\vert \mathcal{P}_N^{(M)}(z) -1 \right\vert \leq \sum_{k=1}^N \binom{N}{k}\left(\frac{\vert z\vert^{2k}}{k!}\right)^M \leq C_{N}(1-\epsilon)^{M}=\mathcal{O}\left(q^M\right),
\end{equation*} 
 as $M\to\infty$, where $C_{N}$ is some positive constant depending on $N$ only.
 Next, let us consider the case $\nu +\epsilon\leq \vert z\vert^2\leq \nu+1 -\epsilon $ for some $\nu\in\{1,\ldots,N-1\}$. Then the sequence $\frac{\vert z \vert^{2k}}{k!}, k=0,\ldots,N,$ attains its unique maximum at $k=\nu$. Moreover, for $k\neq\nu$ we can estimate the following quotients 
 for $k<\nu$ by
 \begin{eqnarray}
  \frac{\binom{N}{k} \left(\frac{\vert z \vert^{2k}}{k!}\right)^M }{\binom{N}{\nu} \left(\frac{\vert z \vert^{2\nu}}{\nu!}\right)^M}
  &=&\frac{(N-\nu)! \nu!}{(N-k)!k!}\left(\frac{\nu!}{k!} \left(\frac{1}{\vert z \vert^2}\right)^{\nu-k}\right)^M\leq  \frac{(N-\nu)! \nu!}{(N-k)!k!} \left(\frac{\nu(\nu-1)\cdots(k+1)}{(\nu+\epsilon)^{\nu-k}}\right)^M 
  \nonumber\\
 &\leq& \frac{(N-\nu)! \nu!}{(N-k)!k!}\left(\frac{N}{N+\epsilon}\right)^M=\mathcal{O}\left(q^M\right),\nonumber
 \end{eqnarray}
 as $M\to\infty$, and for $k>\nu$ in a similar way by
 \begin{align*}&\frac{(N-\nu)! \nu!}{(N-k)!k!}\left(\frac{\nu!}{k!} \vert z\vert^{2(k-\nu)}\right)^M\leq  \frac{(N-\nu)! \nu!}{(N-k)!k!} \left(\frac{(\nu+1-\epsilon)^{k-\nu}}{k(k-1)\cdots(\nu+1)}\right)^M\\
 &\leq 
 \frac{(N-\nu)! \nu!}{(N-k)!k!} \left(\frac{(\nu+1-\epsilon)^{k-\nu}}{(\nu+1)^{k-\nu}}\right)^M
 \leq \frac{(N-\nu)! \nu!}{(N-k)!k!}\left(\frac{N-\epsilon}{N}\right)^M\\
 &\leq \frac{(N-\nu)! \nu!}{(N-k)!k!} \left(\frac{N}{N+\epsilon}\right)^M =\mathcal{O}\left(q^M\right),
 \end{align*}
 as $M\to\infty$, uniformly in $\nu +\epsilon\leq \vert z\vert^2\leq \nu+1 -\epsilon $. Hence, in this region we have
 \begin{align*}
 \mathcal{P}_N^{(M)}(z)&=\binom{N}{\nu}(-1)^{\nu}\left(\frac{z^{2\nu}}{\nu!}\right)^M + \sum_{k=0,k\neq\nu}^N\binom{N}{k}(-1)^k \left(\frac{z^{2k}}{k!}\right)^M\\
 &=\binom{N}{\nu}(-1)^{\nu}\left(\frac{z^{2\nu}}{\nu!}\right)^M\left(1+\sum_{k=0,k\neq\nu}^N \frac{\binom{N}{k}(-1)^k \left(\frac{z^{2k}}{k!}\right)^M}{\binom{N}{\nu}(-1)^\nu \left(\frac{z^{2\nu}}{\nu!}\right)^M} \right).
 \end{align*}
The sum can be estimated by
\begin{equation*}
\left\vert \sum_{k=0,k\neq\nu}^N \frac{\binom{N}{k}(-1)^k \left(\frac{z^{2k}}{k!}\right)^M}{\binom{N}{\nu}(-1)^\nu \left(\frac{z^{2\nu}}{\nu!}\right)^M}\right\vert \leq \sum_{k=0,k\neq\nu}^N\frac{(N-\nu)!\nu!}{(N-k)!k!}\left(\frac{N}{N+\epsilon}\right)^M=\mathcal{O}\left(q^M\right)\ , 
\end{equation*}
 as $M\to\infty$, uniformly in $\nu +\epsilon\leq \vert z\vert^2\leq \nu+1 -\epsilon $.
 Furthermore, if $\vert z \vert^2 \geq N+\epsilon $ (including the point at infinity) we have
\begin{equation*}
 \left\vert \frac{P_N^{(M)}(z)}{\binom{N}{N}(-1)^N \left(\frac{z^{2N}}{N!}\right)^M}-1\right\vert\leq \sum_{k=0}^{N-1}\binom{N}{k}\left(\frac{N!}{k!}\left(\frac{1}{\vert z\vert^2}\right)^{N-k}\right)^M,
\end{equation*} 
 which is less or equal than
\begin{equation*}
 \sum_{k=0}^{N-1}\binom{N}{k}\left(\frac{N!}{k!}\left(\frac{1}{N+\epsilon}\right)^{N-k}\right)^M=\sum_{k=0}^{N-1}\binom{N}{k}\left(\frac{N(N-1)\cdots(k+1)}{(N+\epsilon)^{N-k}}\right)^M.
\end{equation*} 
We can estimate this further by
\begin{equation*}
 \sum_{k=0}^{N-1}\binom{N}{k}\left(\frac{N}{N+\epsilon}\right)^M=\mathcal{O}\left(q^M\right),
\end{equation*} 
 as $M\to\infty$, uniformly in $\vert z \vert^2 \geq N+\epsilon $.
 Collecting these asymptotics in all the regions gives the statement in \eqref{strongasymp}.
\end{proof}
It follows from the asymptotic relation \eqref{strongasymp} that all zeros of $\mathcal{P}_N^{(M)}$ for large values of $M$ accumulate near the circles around the origin with radii $1,\sqrt{2},\ldots,\sqrt{N}$. 

Next, we show that every point of these circles indeed is a limit point of the zeros, from which we can deduce that the zeros converge weakly to the uniform distribution on the union of these circles. Due to the specific rescaling of the polynomials it is sufficient to show that every point $1,\sqrt{2},\ldots,\sqrt{N}$ is a limit point of the zeros. To this end we study the behaviour of the polynomials $\mathcal{P}_N^{(M)}$ in the neighbourhood of these points.
\begin{prop}\label{asymptoticlocal}
	Let us consider a fixed $\nu \in \{1,\ldots,N\}$ and a fixed $N$. Then we have for some $q\in (0,1)$
	\begin{align}\label{localasym}
	\mathcal{P}_N^{(M)}\left(\sqrt{\nu + \frac{w}{M}}\right)=&\binom{N}{\nu-1}(-1)^{\nu-1}\left(\frac{\left(\nu+\frac{w}{M}\right)^{\nu-1}}{(\nu-1)!}\right)^M\\\nonumber
	\qquad& \times \left(1-\frac{N+1-\nu}{\nu}\left(1+\frac{w}{\nu M}\right)^M+\mathcal{O}\left(q^M\right)\right),
	\end{align}
	as $M\to\infty$, uniformly in $w$ on compact subsets of the complex plane.
\end{prop}
\begin{proof}
	We have, using the explicit representation of $\mathcal{P}_N^{(M)}$,
	\begin{align*}
&	\mathcal{P}_N^{(M)}\left(\sqrt{\nu + \frac{w}{M}}\right) = \binom{N}{\nu-1}(-1)^{\nu-1}\left(\frac{(\nu+\frac{w}{M})^{\nu-1}}{(\nu-1)!}\right)^{M}+\binom{N}{\nu}(-1)^{\nu}\left(\frac{(\nu+\frac{w}{M})^{\nu}}{\nu!}\right)^{M}\\
	&\quad\quad\quad\quad\quad\quad\quad\quad\quad+\sum_{k=0..N,k\neq \nu, \nu-1}\binom{N}{k}(-1)^k \left(\frac{\left(\nu+\frac{w}{M}\right)^{k}}{k!}\right)^M\\
	=&\binom{N}{\nu-1}(-1)^{\nu-1}\left(\frac{(\nu+\frac{w}{M})^{\nu-1}}{(\nu-1)!}\right)^{M}\\
	&\times\left(1-\frac{N+1-\nu}{\nu}\left(1+\frac{w}{\nu M}\right)^M+\sum_{k=0..N,k\neq \nu, \nu-1}\frac{\binom{N}{k}(-1)^k \left(\frac{\left(\nu+\frac{w}{M}\right)^{k}}{k!}\right)^M}{\binom{N}{\nu-1}(-1)^{\nu-1}\left(\frac{(\nu+\frac{w}{M})^{\nu-1}}{(\nu-1)!}\right)^{M}}\right).
	\end{align*}
	Using the same estimates as in the proof of the statement \eqref{strongasymp} it is not difficult to see that the latter sum in fact is of order $\mathcal{O}\left(q^M\right)$, as $M\to \infty$, uniformly in $w$ on compact subsets of $\mathbb{C}$, which gives the statement in \eqref{localasym}.
\end{proof}

From the statement \eqref{localasym} we can infer that we have
\begin{align*}
\mathcal{P}_N^{(M)}\left(\sqrt{\nu + \frac{w}{M}}\right)
=\binom{N}{\nu-1}(-1)^{\nu-1}\left(\frac{\nu^{\nu-1}}{(\nu-1)!}\right)^Me^{w\left(1-\frac{1}{\nu}\right)}
\left(1-\frac{N+1-\nu}{\nu}e^{\frac{w}{\nu}}+o(1)\right),
\end{align*}
as $M\to\infty$, uniformly in $w$ on compact subsets of the complex plane. From this we observe that, for large $M$, the rescaled polynomials $\mathcal{P}_N^{(M)}\left(\sqrt{\nu + \frac{w}{M}}\right)$  have exactly one simple, real positive zero located approximately at the point 
\begin{equation*}
w=\nu\log \left(\frac{\nu}{N-\nu+1}\right).
\end{equation*}
This means that, for large $M$, the polynomials $\mathcal{P}_N^{(M)}(z)$ have exactly one simple positive zero in the neighbourhood of $\sqrt{\nu}$ located approximately at 
\begin{equation*}
\sqrt{\nu}\left(1+\frac{1}{2M}\log \frac{\nu}{N+1-\nu}\right).
\end{equation*}
Altogether, this shows that the zeros of the polynomials $\mathcal{P}_N^{(M)}(z)$ converge weakly, as $M\to\infty$, to the uniform distribution on the union of the circles around the origin with radii $1,\sqrt{2},\ldots,\sqrt{N}$. The plot below shows the zeros of 
$\mathcal{P}_N^{(M)}(\sqrt{z})$ for $N=5$ and $M=50$ (the argument is changed to $\sqrt{z}$ just for the reason of better visibility, in this case the zeros converge to the points $1,2,\ldots,N$).
\vspace{0.5cm}

\includegraphics[scale=0.5]{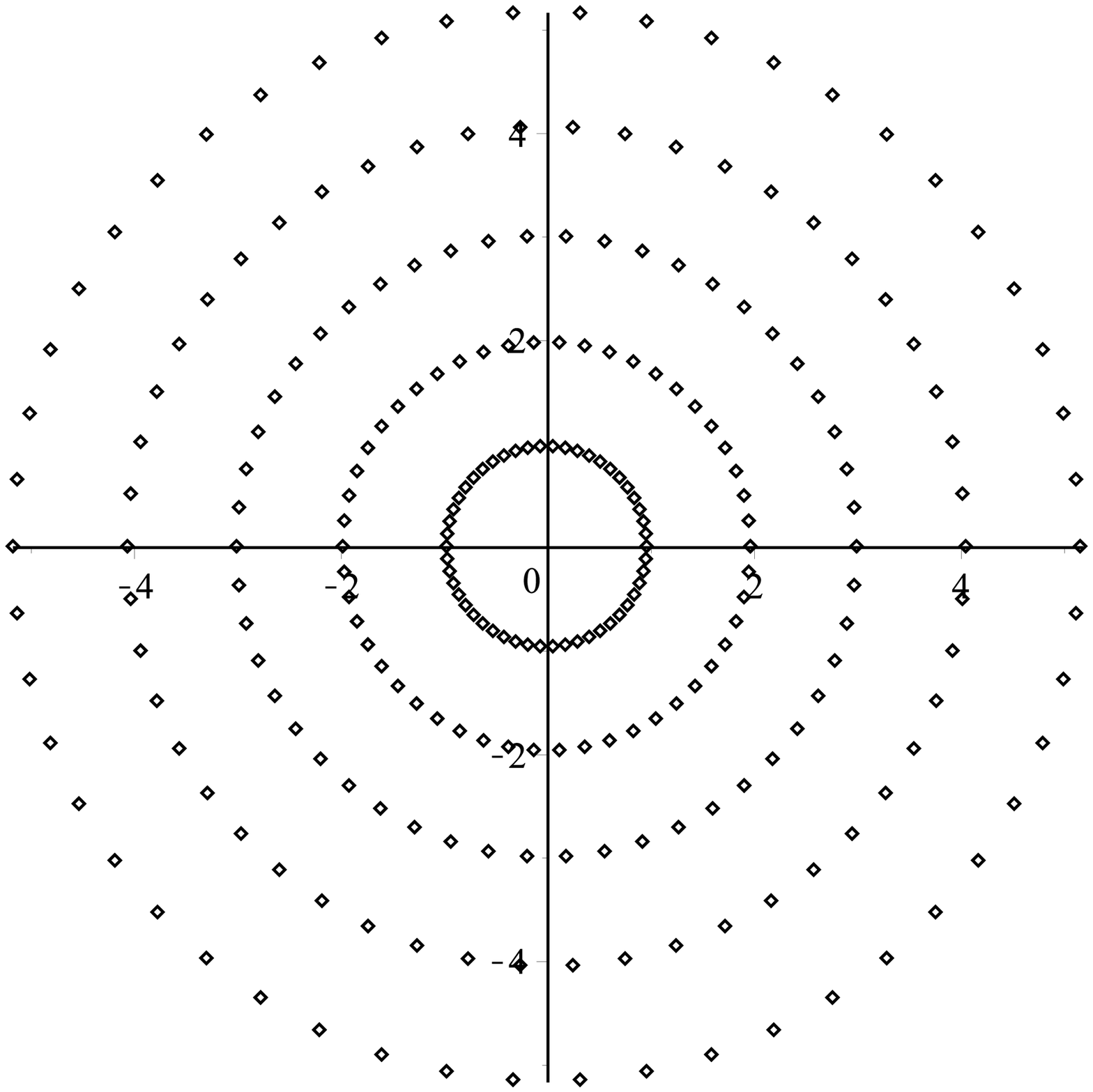}

\vspace{0.5cm}

For the zeros $z_j^{(M)}$ of the average characteristic polynomials 
$p_N^{(M)}$ in \eqref{OP}
this means
\begin{equation*}
\lim_{M\to\infty} \frac{1}{2M}\log \left(z_j^{(M)}\right)= \frac{1}{2}\log( j),\quad j=1,\ldots,N.
\end{equation*}

The above analysis can be used to deal with the case of general variances. To this end, we look at the rescaled polynomials  \eqref{M-Wigner-prod}
\begin{equation}
Q_N^{(M)} (w) :=  \mathcal{P}_N^{(M)} \left(\frac{w}{(\sigma_1 \cdots \sigma_M)^{1/M}}\right)= \sum_{k=0}^N \binom{N}{k} \frac{(-1)^k}{(\sigma^2_1\cdots \sigma^2_M)^k} \left(\frac{w^{2k}}{k!}\right)^{M},\ w\in \mathbb{C}.
\end{equation}
Assuming the additional condition 
\begin{equation*}
\lim_{M\to\infty} \left(\sigma_1 \cdots \sigma_M\right)^{\frac{1}{M}}= \sigma >0\ ,
\end{equation*}
or by simply setting $\sigma_k=\sigma>0$ for all $k$, 
we immediately obtain the asymptotic behaviour of $Q_N^{(M)}(w)$ in corresponding ring domains of the complex plane from the uniformity of Proposition \ref{asymptoticsplane}. Together with a statement analogous to Proposition \ref{asymptoticlocal} we infer that the zeros of the polynomials $Q_N^{(M)}(w)$ converge weakly, as $M\to\infty$, to the uniform distribution on the union of the circles around the origin with radii $\sigma, \sigma \sqrt{2},\ldots,\sigma \sqrt{N}$. This means for the zeros $w_j^{(M)}$ of the average characteristic polynomials in \eqref{M-Wigner-prod} that 
\begin{equation}
\lim_{M\to\infty} \frac{1}{2M}\log \left(w_j^{(M)}\right)= \frac{1}{2}\log (j) + \log (\sigma)\ ,\quad j=1,\ldots,N.
\end{equation}
This proves the statement of Theorem \ref{Thmzeros}.

\section*{Acknowledgments}
Support by the German research council
DFG through the grant  CRC 1283 ``Taming uncertainty and profiting from randomness and low regularity in analysis, stochastics and their applications''
is acknowledged.
We thank Mario Kieburg for useful discussions, and  
the Department of Mathematics at the Royal Institute of Technology (KTH) Stockholm for hospitality (G.A. and T.N.).

%%%%%%%%%%%%%%%%%%%%%%%%%%%%%%%%%%%%%%%%%%%%%%%%%%%5


\begin{thebibliography}{00}

\bibitem{Afanasiev}
Afanasiev, I. On the correlation functions of the characteristic polynomials of non-Hermitian random matrices with independent entries. Journal of Statistical Physics 176(6) (2019) 1561-1582 [arXiv:1902.09390 [math-ph]].

\bibitem{ACi}
Akemann, G.,  Cikovic, M. Products of random matrices from fixed trace and induced Ginibre ensembles. Journal of Physics A: Mathematical and Theoretical 51(18) (2018) 184002 [arXiv:1711.05488 [math-ph]].

\bibitem{AIK}
Akemann, G., Ipsen, J. R.,  Kieburg, M. Products of rectangular random matrices: singular values and progressive scattering. Physical Review E 88(5) (2013) 052118 [arXiv:1307.7560 [math-ph]].

\bibitem{ABu}
Akemann, G.,  Burda, Z. Universal microscopic correlation functions for products of independent Ginibre matrices. Journal of Physics A: Mathematical and Theoretical 45(46) (2012) 465201 [arXiv:1208.0187 [math-ph]].

\bibitem{APS}
Akemann, G., Phillips, M. J.,  Sommers, H. J. Characteristic polynomials in real Ginibre ensembles. Journal of Physics A: Mathematical and Theoretical 42(1) (2008)  012001
[arXiv:0810.1458 [math-ph]].

\bibitem{AV}
Akemann, G.,  Vernizzi, G.  Characteristic polynomials of complex random matrix models. Nuclear Physics B 660(3) (2003) 532-556 [hep-th/0212051].

\bibitem{BDS}
Baik, J., Deift, P., Strahov, E. Products and ratios of characteristic polynomials of random Hermitian matrices. Journal of Mathematical Physics 44(8) (2003)  3657-3670 [math-ph/0304016].

\bibitem{BorodinStrahov}
Borodin, A.,  Strahov, E.  Averages of characteristic polynomials in random matrix theory. Communications on Pure and Applied Mathematics 59(2) (2006) 161-253 [math-ph/0407065].

\bibitem{BorodinOP}
Borodin, A. Biorthogonal ensembles. Nuclear Physics B 536(3)  (1998) 704-732 
[math/9804027].

\bibitem{BH}
Br\'ezin, E.,  Hikami, S.  Characteristic polynomials of random matrices. Communications in Mathematical Physics 214(1) (2000) 111-135 [math-ph/9910005].

\bibitem{BurdaLivanSwiech}
Burda, Z., Livan, G.,  Swiech, A.  Commutative law for products of infinitely large isotropic random matrices. Physical Review E 88(2) (2013) 022107 [arXiv:1303.5360 [cond-mat.stat-mech]].

\bibitem{ForresterDesrosier}
Desrosiers, P.,  Forrester, P. J.  A note on biorthogonal ensembles. Journal of Approximation Theory 152(2) (2008) 167-187 [math-ph/0608052].

\bibitem{Erdoes}
Erd{\H o}s, L. Universality of Wigner random matrices: a survey of recent results. Russian Mathematical Surveys 66(3) (2011) 507 [arXiv:1004.0861 [math-ph]].

\bibitem{FG04}
Forrester, P. J.,  Gamburd, A. Counting formulas associated with some random matrix averages. Journal of Combinatorial Theory, Series A 113(6) (2006) 934-951 
[math/0503002].

\bibitem{FI} 
Forrester, P. J., Ipsen, J. R. Real eigenvalue statistics for products of asymmetric real Gaussian matrices. Linear Algebra and its Applications 510 (2016) 259-290.
[arXiv:1608.04097 [math-ph]].

\bibitem{Peter}
Forrester, P. J. Lyapunov exponents for products of complex Gaussian random matrices. Journal of Statistical Physics 151(5)  (2013) 796-808 [arXiv:1206.2001 [math.PR]].

\bibitem{FS1}
Fyodorov, Y. V.,  Strahov, E. An exact formula for general spectral correlation function of random Hermitian matrices. Journal of Physics A: Mathematical and General 36(12) (2003) 3203 [math-ph/0204051].

\bibitem{GK}
G\"otze, F.,  K\"osters, H. On the second-order correlation function of the characteristic polynomial of a Hermitian Wigner matrix. Communications in Mathematical Physics 285(3) (2009) 1183-1205 [arXiv:0803.0926 [math.PR]].

\bibitem{GGK}
Gr\"onqvist, J., Guhr, T.,  Kohler, H. The k-point random matrix kernels obtained from one-point supermatrix models. Journal of Physics A: Mathematical and General, 37(6) (2004) 2331 [math-ph/0402018].

\bibitem{Ipsen}
Ipsen, J. R.  Products of independent quaternion Ginibre matrices and their correlation functions. Journal of Physics A: Mathematical and Theoretical 46(26) (2013) 265201 [arXiv:1301.3343 [math-ph]].

\bibitem{IK}
Ipsen, J. R.,  Kieburg, M. Weak commutation relations and eigenvalue statistics for products of rectangular random matrices. Physical Review E 89(3) (2014) 032106 [arXiv:1310.4154 [math-ph]].

\bibitem{JesperLyap}
Ipsen, J. R. Lyapunov exponents for products of rectangular real, complex and quaternionic Ginibre matrices. Journal of Physics A: Mathematical and Theoretical, 48(15)  (2015) 155204 [arXiv:1412.3003 [math-ph]].

\bibitem{KN00}
Keating, J. P.,  Snaith, N. C. Random matrix theory and $\zeta(1/2+ it)$. Communications in Mathematical Physics 214(1) (2000) 57-89.

\bibitem{Mario15}
Kieburg, M. Supersymmetry for products of random matrices. 
Acta Physica Polonica B 46 (2015) 1709 [arXiv:1502.00550 [math-ph]].

\bibitem{K}
K\"osters, H., On the occurrence of the sine kernel in connection with the shifted
moments of the Riemann zeta function,
Journal of Number Theory 130(11) (2010) 2596-2609 [arXiv:0803.1141].

\bibitem{Neuschel1} 
Neuschel, T. Plancherel-Rotach formulae for average characteristic polynomials of products of Ginibre random matrices and the Fuss-Catalan distribution. Random Matrices: Theory and Applications 3 (2014) 1450003 
[arXiv:1311.0365 [math.CA]].
	
\bibitem{Newman}
Newman, C. M. The distribution of Lyapunov exponents: Exact results for random matrices. Communications in Mathematical Physics 103(1) (1986) 121-126.

\bibitem{NIST} Olver, F.W., Lozier, D.W., Boisvert, R.F. and Clark, C.W. eds. {\em NIST handbook of mathematical functions.} Cambridge University Press, Cambridge (2010).
	
\bibitem{Tatyana}	
Shcherbina, T. On the correlation function of the characteristic polynomials of the hermitian Wigner ensemble. Communications in Mathematical Physics 308(1) (2011) 1 [arXiv:1006.2536 [math-ph]].
	
\bibitem{Misha}	
Stephanov, M. A. Random matrix model of QCD at finite density and the nature of the quenched limit. Physical Review Letters 76(24) (1996) 4472-4475 
[hep-lat/9604003].

\bibitem{FS}
Strahov, E.,  Fyodorov, Y. V. Universal results for correlations of characteristic polynomials: Riemann-Hilbert approach. Communications in Mathematical Physics 241(2-3) (2003) 343-382 [math-ph/0210010].

\bibitem{TaoVu}	
Tao, T.,  Vu, V.  Random matrices: universality of local spectral statistics of non-Hermitian matrices. The Annals of Probability 43(2) (2015) 782-874 [arXiv:1206.1893 [math.PR]].	
	
\bibitem{reviewL} 
Viana, M. {\em Lectures on Lyapunov Exponents}, Cambridge University Press, Cambridge (2014).
 		
\bibitem{PZJ}	
Zinn-Justin, P.  Universality of correlation functions of Hermitian random matrices in an external field. Communications in Mathematical Physics 194(3) (1998) 631-650 [cond-mat/9705044].
	
\end{thebibliography}
\end{document}